\documentclass[11pt,a4paper,twoside]{article}
\usepackage{amsmath, amsthm, amscd, amsfonts, amssymb, graphicx, color}
\usepackage[bookmarksnumbered, colorlinks, plainpages]{hyperref}
\input{mathrsfs.sty}

\newtheorem{theorem}{Theorem}[section]
\newtheorem*{thma}{Theorem A}
\newtheorem*{thmb}{Theorem B}

\newtheorem*{lema}{Lemma A}

\newtheorem{lemma}[theorem]{Lemma}

\theoremstyle{definition}

\theoremstyle{remark}

\newtheorem{remark}[theorem]{Remark}
\begin{document}
\title{\large{AN APPLICATION OF COHN'S RULE TO CONVOLUTIONS\\ OF UNIVALENT HARMONIC MAPPINGS}}
\author{Raj Kumar\,$^a$ \thanks{rajgarg2012@yahoo.co.in}\,\,,\, Sushma Gupta\,$^a$, Sukhjit Singh\,$^a$ and {Michael Dorff\,$^b$\thanks{mdorff@math.byu.edu} }\\
\emph{ \small $^a$\,Sant Longowal Institute of Engineering and Technology, Longowal-148106 (Punjab), India.}\\
\emph{ \small$^b$\, Department of Mathematics, Brigham Young University, Provo, Utah, 84602, USA.}}
\date{}
\maketitle
\begin{abstract}
  Dorff et al. [\ref{do and no}], proved that the harmonic convolution of right half-plane mapping with dilatation $-z$ and mapping $f_\beta=h_\beta+\overline{g}_{\beta}$, where $f_\beta$ is obtained by shearing of analytic vertical strip mapping, with dilatation $e^{i\theta}z^n,\,n=1,2,\, \theta\in\mathbb{R},$ is in $S_H^0$ and is convex in the direction of the real axis. In this paper, by using Cohn's rule,  we generalize this result by considering dilatations $\displaystyle {(a-z)}/{(1-az)}$  and  $e^{i\theta}z^n( n \in \mathbb{N},\, \theta\in\mathbb{R}$) of right half-plane mapping and $f_\beta$, respectively.
\end{abstract}

\vspace{1mm}
{\small

}

\section{Introduction}
  Let $S_H$ denotes the class of all harmonic, sense-preserving  and univalent mappings $f=h+\overline g$ in the unit disk  $E = \{z: |z|<1\}$,  which are normalized by the conditions $ f(0)=0 $ and $f_{z}(0)=1$. Therefore a function $f=h+\overline g$ in the class $S_H$  has the representation, \begin{equation} f(z) = z+ \sum _{n=2}^{\infty} a_nz^n + \sum _{n=1}^{\infty}\overline{ b}{_n}\overline{z}^n, \end{equation} for all $z$ in $E$. Denote by $S_H^0$  the subclass of $S_H$ whose functions satisfy an additional condition of normalization $f_{\overline z}(0)=0$. Lewy's Theorem asserts that a harmonic mapping $f=h+\overline g$ defined in $E$, is locally univalent and sense-preserving if and only if the Jacobian of the mapping, defined by $J_f=|h'|^2-|g'|^2,$ is positive or equivalently, if and only if $h'\not=0$ in $E$ and the dilatation function $\omega$ of $f$, defined by $\displaystyle\omega={g'}/{h'}$, satisfies $|\omega(z)|<1$, for all $z\in E$.  Further, let $K_H(K_H^0)$ be the subclass of  $S_H(S_H^0)$ consisting of functions which map the unit disk $E$ onto convex domains. A domain $\Omega$ is said to be convex in a direction $\phi,\, 0 \leq \phi < \pi,$ if every line parallel to the line joining $0$ and $ e ^{i \phi}$ has a connected intersection with $\Omega$.

    Kumar et al. [\ref{ku and do}], defined the harmonic mappings in the right half-plane $F_a=H_a+\overline{G}{_a}\in K_H$, given by $\displaystyle H_a(z)+G_a(z)={z}/{(1-z)}$  with\,dilatation\, $\displaystyle \omega_a(z)={(a-z)}/{(1-az)},\,a\in (-1,1).$ Then, by using the shearing technique due to Clunie and Sheil-Small [\ref{cl and sh}], we get,  \begin{equation}
  \displaystyle H_a(z) =\frac{\frac{1}{1+a}z-\frac{1}{2}z^2}{(1-z)^2}\quad {\rm and}\quad\displaystyle G_a(z) =\frac{\frac{a}{1+a}z-\frac{1}{2}z^2}{(1-z)^2}.
  \end{equation}
Let $f_\beta=h_\beta+\overline{g}{_\beta}$ be the collection of harmonic mappings obtained by shearing of analytic vertical strip mapping \begin{equation}\displaystyle h_\beta(z)+g_\beta(z)=\frac{1}{2i\sin\beta}\log\left(\frac{1+ze^{i\beta}}{1+ze^{-i\beta}}\right),0<\beta<\pi.\end{equation}
For more details about this map see Dorff [\ref{do 2}].\\
\indent Harmonic Convolution (Hadamard product) of two harmonic functions $F(z) = H(z) + \overline G(z)=z + \sum_{n=2}^\infty A_n z^n  +
   \sum_{n=1}^\infty{\overline B}{_n} {\overline z} {^n}$ and
$ f(z)= h(z)+ \overline g(z)  =z+\sum_{n=2}^\infty a_n z^n  +   \sum_{n=1}^\infty\overline
   {b}{_n}\overline{z}{^n}$ in $S_H$ is defined as,
$$ \begin{array}{clll}
(F {\ast} f)(z)&=&  (H {\ast} h)(z) +\overline{(G {\ast} g)}(z)\\
 &= & z+ \sum_{n=2}^\infty a_n A_n z^n + \sum_{n=1}^\infty\overline {{b}}{_n}\overline{{B}}{_n}\overline {z}{^n}.
\end{array} $$

It is well known that unlike the case of conformal mappings, harmonic convolution of two mappings from the class $K_H^0$ may not be in $K_H^0,$ may not be even in $S_H^0$. Also, if $f_1\in K_H^0$ and $f_2\in S_H^0$, then $f_1\ast f_2$ is not necessarily in $S_H^0$. These facts generated a lot of interest in the study of harmonic convolutions of univalent harmonic mappings and a good number of papers recently appeared in the literature on this topic (for example see [\ref{do}, \ref{do and no}, \ref{go}, \ref{ku and do}, \ref{li and po}, \ref{li and po 2}]). In particular, Dorff [\ref{do}] proved the following result:

\begin{thma} Let $f=h+\overline{g}$  with  $h(z)+g(z)=\displaystyle{z}/{(1-z)}$  be the right half-plane mapping and  $f_\beta=h_\beta+\overline{g}{_\beta}\in K_H^0$ be given by (3) with $ \displaystyle {\pi}/{2}\leq\beta<\pi$. Then $f\ast f_\beta \in S^0_H$ and is convex in the direction of the real axis, provided $f\ast f_\beta$ is locally univalent and sense-preserving.\end{thma}

Recently, Dorff et al. [\ref{do and no}] obtained the following result which shows that the condition of convolution being locally univalent and sense-preserving as required in Theorem A, can be dropped in some special cases but not in general.

\begin{thmb} If $f=h+\overline{g},$ where $h(z)+g(z)=\displaystyle{z}/{(1-z)}$  is a right half-plane mapping with dilatation  $\omega(z)=-z$ and $f_\beta=h_\beta+\overline{g}{_\beta}$ is the harmonic map as defined in Theorem A above with dilatation $\omega_n(z)=e^{i\theta}z^n$; then for $n=1,2,$ $f\ast f_\beta \in S^0_H$ and is convex in the direction of the real axis.\end{thmb}

 Proceeding as in Remark 1 in [\ref{do and no}] one can verify that Theorem B does not hold true for $n\geq3.$ Aim of the present paper is to study convolutions of more general harmonic mappings $F_a=H_a+\overline{G}{_a},$ where $H_a$ and $G_a$ are defined by (2) and the mappings $f_\beta=h_\beta+\overline{g}{_\beta}$, where $h_\beta$ and $g_\beta$ are the solutions of (3), with dilatation $\displaystyle \omega_n(z)={g_\beta'(z)}/{h_\beta'(z)}=e^{i\theta}z^n, \,n\in\mathbb{N}.$ We shall find the range of real constant $a$ such that $F_a\ast f_\beta\in S_H^0$ and is convex in the direction of the real axis even if $n\geq3$. Cohn's rule stated below shall play a central role in the proofs of our results in this paper.
\begin{lema} (Cohn's Rule [\ref{ra and sc}, p.375])  Given a polynomial $$t(z)= a_0 + a_1z + a_2z^2+...+a_nz^n$$ of
degree $n$, let $$ t^*(z)=\displaystyle z^n\overline{t\left(\frac{1}{\overline z}\right)} = \overline {a}_n + \overline {a}_{n-1}z + \overline {a}_{n-2}z^2 +...+ \overline{a}_0z^n.$$
Denote by $r$ and $s$ the number of zeros of $t$ inside and on the unit circle $|z|=1$, respectively.
 If $|a_0|<|a_n|,$ then $$ t_1(z)= \frac{\overline {a}_n t(z)-a_0t^*(z)}{z}$$ is of degree $n-1$ and has $r_1=r-1$ and $s_1=s$ number
 of zeros inside and on the unit circle $|z|=1$, respectively.\end{lema}
 \section{Main results}
In addition to Cohn's rule stated above, we begin by proving following two results which are also essential for the proofs of our main results.
\begin{lemma} Let $f_\beta= h_\beta+\overline{g}{_\beta}\,\in K_H^0$ be given by (3) with dilatation $\displaystyle \omega={g_\beta'}/{h_\beta'}$ and $F_a=H_a+\overline{G}{_a}$ be a mapping in the right half-plane defined by (2). Then $\widetilde{\omega}$, the dilatation of $F_a\ast f_\beta$, is given by \begin{equation} \displaystyle \widetilde{\omega}(z)=\left[\frac{2\omega(1+\omega)(a+azcos\beta+zcos\beta+z^2)-z\omega'(1-a)(1+2zcos\beta+z^2)}{2(1+zcos\beta+azcos\beta+az^2)(1+\omega)-z\omega'(1-a)(1+2zcos\beta+z^2)}\right].\end{equation}\end{lemma}
\begin{proof} Since  $$F_a\ast f_\beta= H_a\ast h_\beta +\overline{G{_a}\ast g{_\beta}}=H+\overline{G}\,(say),$$
so $$H(z)=\frac{1}{2}\left[h_\beta+\frac{(1-a)}{(1+a)}zh_\beta'\right] $$ and $$ G(z)=\frac{1}{2}\left[g_\beta-\frac{(1-a)}{(1+a)}zg_\beta'\right].$$

 If $\widetilde{\omega}$ is the dilatation of $F_a\ast f_\beta$, then  \begin{equation}\displaystyle \widetilde{\omega}(z)=\frac{G'(z)}{H'(z)}=\left[\frac{2ag_\beta'-(1-a)zg_\beta''}{2h_\beta'+(1-a)zh_\beta''}\right].\end{equation}
  Since \\
  $\displaystyle g_\beta'=\omega h_\beta'$, \quad $\displaystyle h_\beta'(z)=\frac{1}{(1+\omega)(1+ze^{i\beta})(1+ze^{-i\beta})}$\,\,
 and \,\, $$\indent\hspace{.5cm}\displaystyle h_\beta''(z)=-\frac{2(cos\beta+z)(1+\omega)+\omega'(1+2zcos\beta+z^2)}{(1+\omega)^2(1+ze^{i\beta})^2(1+ze^{-i\beta})^2},$$ therefore, from (5) we get
 $$\displaystyle \widetilde{\omega}(z)=\left[\frac{2a\omega h_\beta'-(1-a)z(\omega h_\beta''+\omega'h_\beta')}{2h_\beta'+(1-a)zh_\beta''}\right].$$
Substituting the values of $h_\beta'$ and $h_\beta''$ and simplifying, we get (4).
\end{proof}
\begin{lemma} Following inequalities hold true:

\noindent(a)$\left|-2\cos\beta+4e^{-i\theta}\cos^2\beta-3e^{-i\theta}+e^{i\theta}\right|\leq\left|4-2\cos^2\beta-2\cos\beta\cos\theta\right|$ for $\displaystyle \beta \in (0,{\pi})$ and $\theta\in \mathbb{R}$.\\
(b) $\displaystyle | \cos\beta(e^{-i\theta}-5)|<|6-\cos^2\beta+2e^{-i\theta}-3e^{-i\theta}\cos^2\beta|$  for $\displaystyle \beta\in(0,{\pi})$ and $\theta\in \mathbb{R}.$\\
(c)  $\displaystyle |2(1+e^{-i\theta})-3e^{-i\theta}\cos^2\beta|<|4-\cos^2\beta|$ for $\displaystyle \beta\in \left(0,{\pi}/{2}\right)\cup\left({\pi}/{2},\pi\right)$ and $\theta\not=2n\pi, n\in\mathbb{Z}$.
\end{lemma}
\begin{proof}  Let $x=\cos\beta$ and $y=\cos\theta$. Then
$$\indent\hspace{-3.5cm}\textbf{(a)} \left|-2\cos\beta+4e^{-i\theta}\cos^2\beta-3e^{-i\theta}+e^{i\theta}\right|^2-\left|4-2\cos^2\beta-2\cos\beta \,\cos\theta\right|^2$$
$\indent\hspace{2cm}=\displaystyle 12x^4-12x^2-24x^3y+24xy+12x^2y^2-12y^2$\\
$\indent\hspace{2cm}=\displaystyle-12(1-x^2)(x-y)^2$\\
$\indent\hspace{2cm}\leq0, \, {\rm as}\,\, x\in (-1,1)\,\,{\rm and}\,\,y\in[-1,1].$

\noindent \textbf{(b)} $\displaystyle|\cos\beta(e^{-i\theta}-5)|^2-|6-\cos^2\beta+2e^{-i\theta}-3e^{-i\theta}\cos^2\beta|^2$
$$ \begin{array}{clll}
\vspace{.1cm}
\hspace{-1.5cm}=-2(5x^4+3x^4y-15x^2y-25x^2+12y+20)\\
\hspace{-4.1cm}=-2(x^2-4)(x^2-1)(5+3y)\\
\hspace{-3.3cm}<0, \,{\rm for\,} x\in(-1,1)\,{\rm and}\, y\in[-1,1].
\end{array} $$

\indent\hspace{-.5cm}\textbf{(c)} \,$\displaystyle |2(1+e^{-i\theta})-3e^{-i\theta}\cos^2\beta|^2-|4-\cos^2\beta|^2$
$$\hspace{-4.2cm}=4(2x^4-x^2-3yx^2+2y-2)$$
$$\hspace{-3cm}=4[2(x^2-1)(x^2+1-y)-x^2(y+1)]$$
$\hspace{3.3cm}<0, \, {\rm for}\,\, x\in (-1,0)\cup(0,1)\,{\rm and}\,\, y\in [-1,1).$
\end{proof}
We are now in a position to prove our main results.
\begin{theorem} Let $f_\beta= h_\beta+\overline{g}{_\beta}\,\in K_H^0$ be a mapping given by (3) with dilatation $\omega_1(z)=e^{i\theta}z.$ If $F_a$ defined by (2) is a mapping in the right half-plane, then $F_a\ast f_\beta \, \in S_H^0$ and is convex in the direction of the real axis for $a \in [-{1}/{3},1)$.\end{theorem}
\begin{proof} Let  $\widetilde{\omega}$ be the dilatation of $F_a\ast f_\beta.$ We claim that $\displaystyle|\widetilde{\omega}(z)|<1$ for all $z\in E$ i.e., $F_a\ast f_\beta$ is locally univalent and sense-preserving. Our result shall then follow from Theorem A. By setting $\omega(z)=\omega_1(z)=e^{i\theta}z$ in (4) we get,
 $$ \begin{array}{clll}
\vspace{.5cm}
\indent\hspace{-.6cm}\widetilde{\omega}(z)=\displaystyle z\left[\frac{e^{i\theta}z^3+(\frac{1}{2}+ae^{i\theta}\cos\beta+e^{i\theta}\cos\beta+\frac{a}{2})z^2+a(2\cos\beta+e^{i\theta})z+\frac{(3a-1)}{2}}{e^{-i\theta}+(e^{-i\theta}\cos\beta +ae^{-i\theta}\cos\beta+\frac{1}{2}+\frac{a}{2})z+a(e^{-i\theta}+2\cos\beta)z^2+\frac{(3a-1)}{2}z^3}\right]\\
\indent\hspace{-11.7cm}=\displaystyle z\frac{p(z)}{p^*(z)},
\end{array} $$
 where,\\
$\indent\hspace{.1cm}p(z)=e^{i\theta}z^3+(\frac{1}{2}+ae^{i\theta}\cos\beta+e^{i\theta}\cos\beta+\frac{a}{2})z^2+a(2\cos\beta+e^{i\theta})z+\frac{(3a-1)}{2}$\\
$\indent\hspace{.8cm}=a_3z^3+a_2z^2+a_1z+a_0$ \\and\\
$\indent p^*(z)=e^{-i\theta}+(e^{-i\theta}\cos\beta+ae^{-i\theta}\cos\beta+\frac{1}{2}+\frac{a}{2})z+a(e^{-i\theta}+2\cos\beta)z^2+\frac{(3a-1)}{2}z^3=z^3\overline{p\left(\frac{1}{\overline z}\right)}$.\\
 Thus if $z_0$, $z_0\not=0,$ is a zero of $p$ then ${1}/{\overline{z}_{0}}$ is a zero of $p^*$. Therefore we can write $$\displaystyle\widetilde{\omega}(z)=z\frac{(z+A)(z+B)(z+C)}{(1+\overline {A}z)(1+\overline {B}z)(1+\overline {C}z)}.$$   In order to prove that $\displaystyle|\widetilde{\omega}(z)|<1$, it suffices to show that $|A|\leq 1$, $|B|\leq1$ and $|C|\leq1$ or equivalently, all the zeros,  $-A,-B$ and $-C,$ of the polynomial $p$ lie in or on the unit circle $|z|=1$ for $a\in[{-1}/{3},1)$. When $a=-{1}/{3},$ then $\displaystyle |\widetilde{\omega}(z)|=|-ze^{i\theta}|<1$ for all $z \in E$, so we take $a\in(-{1}/{3},{1}/{3})\cup({1}/{3},1)$ and the case when $\displaystyle a={1}/{3}$ will be settled separately.\\ Let\\
$\indent\hspace{.1cm}\displaystyle p_1(z)=\frac{\overline{a}_{3}p(z)-a_0p^*(z)}{z}$
\begin{equation}
\indent\hspace{-3.1cm}=\frac{1}{4}(1+2a-3a^2)[3z^2+2(2\cos\beta+e^{-i\theta})z+(2\cos\beta e^{-i\theta}+1)]\\
\end{equation}
$$ \begin{array}{clll}
\indent\hspace{-7.3cm}=b_2z^2+b_1z+b_0.
\end{array} $$
 As $|a_0|<|a_3|$ and $1+2a-3a^2\not=0$ for $a\in(-{1}/{3},{1}/{3})\cup({1}/{3},1)$, therefore by Cohn's rule polynomial $p_1$ has one less number of zeros inside $|z|=1$ than $p$ and same number of zeros on $|z|=1$ as $p$.
 Again, let\\
$\indent\hspace{-.1cm}\displaystyle p_2(z)=\frac{\overline{b}_{2}p_1(z)-b_0p_1^*(z)}{z}$\\
$\indent\hspace{.7cm}=\frac{1}{16}(1+2a-3a^2)^2[(9-|2\cos\beta+e^{i\theta}|^2)z+6(2\cos\beta+e^{-i\theta})-2e^{-i\theta}(2\cos\beta+e^{i\theta})^2],$
\\ where $p_1^*(z)=z^2\overline{p_1\left(\frac{1}{\overline z}\right)}$. We can again apply Cohn's rule on $p_1$ because\\ $|b_0|=|(1+2a-3a^2)\frac{1}{4}(2\cos\beta e^{-i\theta}+1)|<|(1+2a-3a^2)\frac{3}{4}|=|b_2|.$ If the zero of $p_2$ is denoted by $z_0,$ then $$z_0=\frac{-2\cos\beta+4e^{-i\theta}\cos^2\beta-3e^{-i\theta}+e^{i\theta}}{4-2\cos^2\beta-2\cos\beta \,\cos\theta}$$ and in view of Lemma 2.2 (a), $|z_0|\leq1$. Thus we have proved that all the zeros of $p_1$ and therefore of $p$ lie inside or on the unit circle $|z|=1$ for $a\in({-1}/{3},{1}/{3})\cup({1}/{3},1).$ In case when $a={1}/{3},$ then $$p(z)=\frac{1}{3}e^{i\theta}z\left[3z^2+2(2\cos\beta+e^{-i\theta})z+(2\cos\beta e^{-i\theta}+1)\right].$$
In view of (6) it can be easily verified that all the zeros of $p$ lie inside or on the circle $|z|=1$ for $a={1}/{3}$. This completes the proof.
\end{proof}
\begin{remark} By setting $a=0$ in Theorem 2.3, we get Theorem B (for $n=1$), stated in Section 1.\end{remark}

\begin{theorem} If $f_\beta\in K_H^0$ is a mapping with dilatation $\omega_2(z)=e^{i\theta}z^2$, then $F_a\ast f_\beta \, \in S_H^0$ and is convex in the direction of the real axis for $a \in [0,1)$, where $F_a$ is defined by (2).\end{theorem}

\begin{proof}  The case when $a=0$ has already been proved by Dorff et al. [\ref{do and no}], so we take $a\in(0,1)$. Now, if $\widetilde{\omega}$ is the dilatation of $ F_a\ast f_\beta$, then letting $\omega(z)=\omega_2(z)=e^{i\theta}z^2$ in (4) we get,
$$\displaystyle\widetilde{\omega}(z)=z^2e^{2i\theta}\left[\frac{z^4+cos\beta(a+1)z^3+a(1+e^{-i\theta})z^2+e^{-i\theta}cos\beta(3a-1)z+e^{-i\theta}(2a-1)}{1+cos\beta(a+1)z+a(1+e^{i\theta})z^2+e^{i\theta}cos\beta(3a-1)z^3+e^{i\theta}(2a-1)z^4}\right].$$
\noindent In view of Theorem A, we need only to show that $F_a\ast f_\beta$ is locally univalent and sense-preserving, i.e., $\displaystyle\left|\widetilde{\omega}\right|<1$ in $E$. Consider \\
$\indent\hspace{.2cm}q(z)=z^4+\cos\beta(a+1)z^3+a(1+e^{-i\theta})z^2+e^{-i\theta}\cos\beta(3a-1)z+e^{-i\theta}(2a-1)$\\
$\indent\hspace{.9cm}=a_4z^4+a_3z^3+a_2z^2+a_1z+a_0$,\\ and \\
$ \indent\hspace{0.0cm}q^*(z)=z^4\overline{q\left(\frac{1}{\overline z}\right)}$\\
$\indent\hspace{.9cm}=1+\cos\beta(a+1)z+a(1+e^{i\theta})z^2+e^{i\theta}\cos\beta(3a-1)z^3+e^{i\theta}(2a-1)z^4.$ \\

Then $\displaystyle\widetilde{\omega}(z)=z^2e^{2i\theta}\frac{q(z)}{q^*(z)}=z^2e^{2i\theta}\frac{(z+A)(z+B)(z+C)(z+D)}{(1+\overline {A}z)(1+\overline {B}z)(1+\overline {C}z)(1+\overline {D}z)},$ \\

where $-A,-B,-C$ and $-D$ are the zeros of $q$. We shall show that these zeros lie inside or on the unit circle $|z|=1$ for $a\in(0,1)$. First we take $a\in(0,{1}/{2})\cup({1}/{2},1)$ and the case when $\displaystyle a=1/2$  will be dealt separately. For $a\in(0,{1}/{2})\cup({1}/{2},1),$ we have $|a_0|=|2a-1|<1=|a_4|$, so we can apply Cohn's rule on $q$. Let\\

$\indent\hspace{-.7cm}\displaystyle q_1(z)=\frac{\overline{a}_{4}q(z)-a_0q^*(z)}{z}$
\begin{equation}
\indent\hspace{-4.5cm}=2a(1-a)\left(2z^3+3\cos\beta z^2+(1+e^{-i\theta})z+e^{-i\theta}\cos\beta\right).\end{equation}
$\indent\hspace{.7cm}=b_3z^3+b_2z^2+b_1z+b_0.$\\
It is easy to verify that Cohn's rule is applicable to $q_1$ also. So, let\\
$\displaystyle q_2(z)=\frac{\overline{b}_{3}q_1(z)-b_0q_1^*(z)}{z}$\\
$\indent\hspace{.3cm}\displaystyle=4(a(1-a))^2\left[(4-\cos^2\beta)z^2+\cos\beta(5-e^{-i\theta})z+2(1+e^{-i\theta})-3e^{-i\theta}\cos^2\beta\right]$\\
$\indent\hspace{.3cm}\displaystyle=c_2z^2+c_1z+c_0.$\\
To apply Cohn's rule on $q_2$ we need $|2(1+e^{-i\theta})-3e^{-i\theta}\cos^2\beta|< |4-\cos^2\beta|,$ which is true in view of Lemma 2.2 (c) provided  $\beta\not=\displaystyle \pi/2$ and $\theta\not=2n\pi,n\in\mathbb{Z}.$ Therefore by applying Lemma A again on $q_2$, we have\\
$\displaystyle q_3(z)=\frac{\overline{c}_{2}q_2(z)-c_0q_2^*(z)}{z}$\\
$\indent\hspace{.3cm}=16(a(1-a))^4\{[(4-\cos^2\beta)^2-\left(2(1+e^{-i\theta})-3e^{-i\theta}\cos^2\beta\right)^2]z+$\\
$\indent\hspace{4.4cm}\cos\beta(5-e^{-i\theta})\left[(4-\cos^2\beta)-2(1+e^{-i\theta})+3e^{-i\theta}\cos^2\beta\right]\}.$\\
The only zero of $q_3$ shall lie inside or on the circle $|z|=1$, provided $$\displaystyle | \cos\beta(e^{-i\theta}-5)|<|6-\cos^2\beta+2e^{-i\theta}-3e^{-i\theta}\cos^2\beta|, $$ which is true in view of Lemma 2.2(b). Thus by Cohn's rule all zeros of $q_2$, $q_1$ and therefore of $q$ lie in or on the unit circle $|z|=1$ for $\beta\not=\displaystyle \pi/2$ and $\theta\not=2n\pi,n\in\mathbb{Z}$.\\
 \indent In case $\beta=\displaystyle \pi/2$ and $\theta=2n\pi,n\in\mathbb{Z}$, we have $\displaystyle q_2(z)=16(a(1-a))^2 (z^2+1)$ and obviously, zeros of $q_2$ lie on $|z|=1$ because $a\in(0,{1}/{2})\cup({1}/{2},1)$. Therefoer by Cohn's rule all zeros of $q_1$ and of $q$ lie in or on the unit circle $|z|=1$ in this case also. \\
 \indent Now when $\displaystyle a= {1}/{2}$, then we have $$ q(z)=\frac{z}{2}\left(2z^3+3\cos\beta z^2+(1+e^{-i\theta})z+e^{-i\theta}\cos\beta\right).$$ Keeping (7) in view, we can easily verify that in this case all zeros of $q$ lie in or on the unit circle $|z|=1$. Hence $\displaystyle\left|\widetilde{\omega}\right|<1$ for $a\in [0,1)$. \end{proof}
As pointed out in Section 1, Theorem B is not true for $n\geq3,$ however following theorem, which we state without proof (as proof runs on the same lines as of Theorem 2.5), asserts that the conclusion of Theorem 2.5 remains valid when dilatation of $f_\beta$ is taken as $e^{i\theta}z^3$ and value of real constant $a$ is restricted in the interval [1/5, 1).
\begin{theorem} Let $f_\beta\in K_H^0$ be a harmonic map with dilatation $\omega_3(z)=e^{i\theta}z^3$. Then $F_a\ast f_\beta \, \in S_H^0$ and is convex in the direction of the real axis for $ a \in [{1}/{5},1)$, where $F_a$ is defined by (2).\end{theorem}
 \begin{remark} It is expected that, keeping $\omega_n(z)=e^{i\theta}z^n\,(\theta\in\mathbb{R}, n\in\mathbb{N})$ as the dilatation of $f_\beta$, the convolution  $F_a\ast f_\beta \, \in S_H^0$ and is  convex in the direction of the real axis for $ a \in [\frac{n-2}{n+2},1).$ The authors have verified the case when $n=4$ but when $n$ increases calculations become extremely cumbersome.  \end{remark}
\noindent{\emph{Acknowledgement: The first author is thankful to the Council of Scientific and Industrial Research, New Delhi, for financial support vide grant no. 09/797/0006/2010 EMR-1.}}
{

\end{document}